\documentclass[11pt]{amsart} \usepackage{latexsym, amssymb}
\usepackage[T1]{fontenc}

\newtheorem{thm}{Theorem}[section]
\newtheorem{lemma}[thm]{Lemma}
\newtheorem{prop}[thm]{Proposition}

\theoremstyle{definition}
\newtheorem{df}[thm]{Definition}
\newtheorem{nrmk}[thm]{Remark}

\theoremstyle{remark}

\renewcommand{\r}{\mathbb{R}}
\newcommand{\Z}{\mathbb{Z}}

\newcommand{\curly}[1]{\mathcal{#1}}

\newcommand{\la}{\curly{L}}

\renewcommand{\to}{\rightarrow}

\def \<{\langle}
\def \>{\rangle}

\def \*Z {{{^*}\Z}}
\def \((  {(\!(}
\def \)) {)\!)}

\def \tp{\operatorname{tp}}

\numberwithin{equation}{section}

\def \u{\mathcal U}

\makeatletter

\def\indsym#1#2{%
  \setbox0=\hbox{$\m@th#1x$}%
  \kern\wd0%
  \hbox to 0pt{\hss$\m@th#1\mid$\hbox to 0pt{$\m@th#1^{#2}$}\hss}%
  \lower.9\ht0\hbox to 0pt{\hss$\m@th#1\smile$\hss}%
  \kern\wd0}

\def\nindsym#1#2{%
  \setbox0=\hbox{$\m@th#1x$}%
  \kern\wd0%
  \hbox to 0pt{\hss$\m@th#1\not$\kern1.4\wd0\hss}
  \hbox to 0pt{\hss$\m@th#1\mid$\hbox to 0pt{$\m@th#1^{\,#2}$}\hss}%
  \lower.9\ht0\hbox to 0pt{\hss$\m@th#1\smile$\hss}%
  \kern\wd0}

\allowdisplaybreaks[2]

\title{Dividing and weak quasi-dimensions in arbitrary theories}
\author{Isaac Goldbring and Henry Towsner}
\thanks{Goldbring's work was partially supported by NSF CAREER grant DMS-1349399.}
\address {University of Illinois at Chicago, Department of Mathematics, Statistics, and Computer Science, Science and Engineering Offices (M/C 249), 851 S. Morgan St., Chicago, IL 60607-7045, USA}
\email{isaac@math.uic.edu}
\urladdr{www.math.uic.edu/~isaac}

\address {University of Pennsylvania, Department of Mathematics, 209 S. 33rd Street, Philadelphia, PA 19104-6395}
\email{htowsner@math.upenn.edu}
\urladdr{www.sas.upenn.edu/~htowsner}

\begin{document}

\begin{abstract}
We show that any countable model of a model complete theory has an elementary extension with a ``pseudofinite-like'' quasi-dimension that detects dividing.
\end{abstract}

\maketitle

\section{Introduction}

In a pseudofinite structure, every set $S$ has a size $|S|$, a nonstandard cardinality.  It is reasonable to say that $S$ and $T$ are ``similar in size'' if $|\log |S|-\log |T||$ is bounded (by a natural number).  This gives the notion of \emph{fine pseudofinite dimension} \cite{MR3091666,MR2833482}, the quotient of $\log|S|$ by a suitable convex set.  Garc\'ia shows \cite{2014arXiv1402.5212G} that the fine pseudofinite dimension detects dividing: roughly, if $\phi(x,b)$ divides over $\psi(x,a)$ then there is a $b'$ with $tp(b'/a)=tp(b/a)$ so that the dimension of $\phi(x,b')$ is strictly stronger than the dimension of $\psi(x,a)$.

We give a limited extension of this to model complete theories (and, via Morleyization, to any theory): any countable model whose theory is model complete embeds elementarily in a ``large'' fragment of a pseudofinite structure in such a way that the notion of dimension pulls back to the original model, and if $\phi(x,b)$ divides over $\psi(x,a)$ then there is a $b'$ in the elementary extension with $tp(b'/a)=tp(b/a)$ so that the dimension of $\phi(x,b')$ is strictly stronger than the dimension of $\psi(x,a)$.

Note that it is easy to embed a countable structure in a pseudofinite structure (embed $M$ in the ultraproduct of its finite restrictions), but that this is not an elementary embedding.  It is also easy to obtain a dimension-like function that detects dividing by linearizing the partial order on definable sets given by dividing.  The dimension here, however, is an abelian group, and even a quotient of $\r^*$.

We would like to thank Dario Garc\'ia for useful comments on an earlier draft of this note.

\section{Construction}

Let $\la$ be a countable first-order relational signature and let $T$ be a complete, model complete theory in $\la$.  Set $\la':=\la\cup \{V_\alpha \ : \ \alpha<\omega+\omega\}$, where the $V_\alpha$ are fresh unary relation symbols.  All $\la'$-structures considered will have the property that the interpretations of the $V_\alpha$'s will form a chain.

Fix an enumeration of quantifier free formulas in $\la$ with distinguished choices of free variables $\vec x,\vec y$, $\varphi_i(\vec x,\vec y)$, so that each formula $\varphi_i$ appears infinitely often.  Fix an ordering of $\omega+\omega$ of order-type $\omega$.  Combining these, we obtain an ordering of order-type $\omega$ of pairs $\sigma_i=(\varphi_i(\vec x,\vec y),\alpha_i)$.  We write $\alpha(\sigma_i)$ for $\alpha_i$.

We say that an $\la$-structure $M$ \emph{strongly satisfies} $\sigma_i$ if, whenever $\vec a\in V_{\alpha_i}^{M}$ is such that there is $N\supseteq M$ with $N\models T$ and $\vec b\in N$ such that $N\models \varphi_i(\vec a,\vec b)$ 
then there is a $\vec b\in V_{\alpha_i+1}^{M}$ such that $M\models \varphi_i(\vec a,\vec b)$.

For each $n\in \omega$, we define an $\la'$-structure $M_n\models T_\forall$ with the property that if $i<n$, $M_n$ strongly satisfies $\sigma_i$.  Let $M_0$ denote a one-element substructure of a model of $T$ whose unique element satisfies each $V_n$.  

Suppose we have constructed $M_{n-1}$.  Consider the first $n$ pairs $\sigma_0,\ldots,\sigma_{n-1}$ and fix a permutation $\sigma_{r_0},\ldots,\sigma_{r_{n-1}}$  so that $i\leq j$ implies that $\alpha(\sigma_{r_i})\leq\alpha(\sigma_{r_j})$.  We construct $\la'$-structures $M_n^i\models T_\forall$, for $i=0,\ldots,n$, by recursion on $i$ in such a way that $M_n^i$ strongly satisfies $\sigma_{r_0},\ldots,\sigma_{r_{i-1}}$.  We will then set $M_n:=M_n^n$.

Let $M_n^0=M_{n-1}$.  Suppose that $M_n^i$ has been constructed and set $\alpha:=\alpha(\sigma_{r_i})$ and let $\varphi(\vec x,\vec y)$ be the formula in $\sigma_{r_i}$.  Enumerate the tuples of length $|\vec x|$ in $V_\alpha^{M_n^i}$ as $\vec a_1,\ldots,\vec a_k$.  We now recursively construct a sequences of models $M_n^{i,j}$; we begin with $M_n^{i,0}=M_n^i$.  Given $M_n^{i,j}$, we proceed as follows:
\begin{itemize}
\item If there is a $\vec b\in M_n^{i,j}$ such that $M_n^{i,j}\models \phi(\vec a_j,\vec b)$ and $\vec b$ is contained in $V_{\alpha+1}^{M_n^{i,j}}$ then $M_n^{i,j+1}=M_n^{i,j}$,
\item Otherwise, if there is an extension $M$ of $M_n^{i,j}$ and a $\vec b$ in $M$ such that $M\models\phi(\vec a_j,\vec b)$ then $M_n^{i,j+1}=M_n^{i,j}\cup\{\vec b\}$ and any element of $\vec b$ which is not in $V_{\alpha+1}^{M_n^{i,j}}$ is in $V_{\alpha+1}^{M_n^{i,j+1}}\setminus V_\alpha^{M_n^{i,j+1}}$,
\item Otherwise, set $M_n^{i,j+1}=M_n^{i,j}$.
\end{itemize}
Set $M_n^{i+1}:=M_n^{i,k+1}$.  Note that $M_n^{i+1}$ strongly satisfies $\sigma_{r_0},\ldots, \sigma_{r_i}$ as desired because $V_\alpha^{M_n^{i+1}}=V_\alpha^{M_n^i}$. 

Let $M:=\prod_\u M_n$.  By definition, $V_\alpha^M=\{x\in M\ :\ M\vDash x\in V_\alpha\}$.  We write $V_{<\omega}^M$ for $\bigcup_{n<\omega}V_n^M$.  We define $M'=\{x\in M\mid\exists n\  x\in V_{\omega+n}^M\}$.  Note that $V_\alpha^M=V_\alpha^{M'}$ for all $\alpha$.

Since $T$ is model-complete, it has a set of $\forall\exists$-axioms.  Suppose that $\forall \vec x\exists \vec y\varphi(\vec x,\vec y)$ is such an axiom.  Fix $\alpha<\omega+\omega$ and take $i$ such that $\sigma_i=(\varphi(\vec x,\vec y),\alpha)$.  Fix $n>i$ and consider $N\models T$ such that $M_n\subseteq N$.  Since $N\models T$ and $M_n$ strongly satisfies $\sigma_i$, we have $M_n\models \forall \vec x\in V_\alpha \exists \vec y \in V_{\alpha+1} \varphi(\vec x,\vec y)$.  This proves:


\begin{prop}
$V_{<\omega}^M$ and $M'$ are both models of $T$.
\end{prop}

There is no guarantee that $V_\omega^M$ is a model of $T$.  Note that $M'$ is existentially closed in $M$.  Indeed, $M\models T_\forall$ and every model of $T$ is an existentially closed model of $T_\forall$.

\

\begin{lemma}
Suppose that $p(\vec x)$ is a countable set of existential $\la$-formulae with parameters from $V^M_\alpha$ that is finitely satisfiable in $M'$ (equivalently, in $M$).  Then there is $\vec c\in V^M_{\alpha+1}$ such that $M'\models p(\vec c)$.
\end{lemma}
\begin{proof}
Let $p'(\vec x):=p(\vec x)\cup \{\vec x\in V_{\alpha+1}\}$.  Since $M$ is countably saturated, it suffices to show that $p'(\vec x)$ is finitely satisfiable in $M$.  Indeed, we then get $\vec c\in V_{\alpha+1}^M$ such that $M\models p'(c)$,  Since the formulae in $p$ are existential and, to show that $M\models \exists \vec x\in V_{\alpha+1}\exists \vec z\phi(\vec x,\vec z,\vec a)$ it clearly suffices to show $M\models \exists \vec x\in V_{\alpha+1}\exists \vec z\in V_{\alpha+1}\phi(\vec x,\vec z,\vec a)$, it suffices to consider quantifier-free formulae.

So it suffices to show that, if $\varphi(\vec x,\vec a)$ is a quantifier-free formula with parameters from $V_\alpha^M$ such that $M'\models \exists \vec x \varphi(\vec x,\vec a)$, then $M'\models \exists \vec x\in V_{\alpha+1}\varphi(\vec x,\vec a)$.  Since $M\models \exists \vec x \varphi(\vec x,\vec a)$, we get $M_n\models \exists \vec x \varphi(\vec x,\vec a_n)$ for $\u$-almost all $n$, where $(\vec a_n)$ is a representative sequence for $\vec a$.  Fix $i$ such that $\sigma_i=(\varphi(\vec x,\vec y),\alpha)$.  Since $M_n$ strongly satisfies $\sigma_i$ for $n>i$, it follows that, for $\u$-almost all $n$, we can find $c_n\in V_{\alpha+1}^{M_n}$ such that $M_n\models \varphi(\vec c_n,\vec a_n)$.
\end{proof}

\begin{lemma}
Suppose $N$ is a model of $T$ and $A\subseteq N$ is countable.  Then there is $A'\subseteq V^{M'}_{<\omega}$ such that $\tp^{N}_{\la}(A)=\tp^{M'}_{\la}(A')$.
\end{lemma}
\begin{proof}
Enumerate $A=\{a_0,a_1,\ldots\}$ and construct $A'\subseteq M'$ inductively: given $a'_0,\ldots,a'_n\in V^{M'}_{n+1}$ with $\tp_{\la}^{M'}(a'_0,\ldots,a'_n)=tp_{\la}^{N}(a_0,\ldots,a_n)$, there exists, by the previous lemma, an $a'_{n+1}\in V^{M'}_{n+2}$ with $\tp_{\la}^{M'}(a'_{n+1}/a'_0,\ldots,a'_n)=\tp_{\la}^{N}(a_{n+1}/a_0,\ldots,a_n)$.  (Note that, since $T$ is model-complete, we can replace the complete types with the complete existential types.)
\end{proof}

We now fix a countable model $N$ of $T$ and take $A=N$ in the above lemma, yielding an elementary embedding $a\mapsto a':N\to V_{<\omega}^M$ with image $N'$.

\begin{df}
For an $\la$-formula $\varphi(\vec x)$, we let $\varphi_\omega(\vec x):=\varphi(\vec x)\wedge (\vec x\in V_\omega)$.
\end{df}

\begin{lemma}
Suppose that $\varphi(\vec x,\vec y)$ and $\psi(\vec x,\vec z)$ are  existential $\la$-formulae.  If $N\models \forall \vec x (\varphi(\vec x,\vec a)\leftrightarrow \psi(\vec x,\vec b))$, then $M\models \forall \vec x(\varphi_\omega(\vec x,\vec a')\leftrightarrow \psi_\omega(\vec x,\vec b'))$.
\end{lemma}

\begin{proof}
Since $a\mapsto a'$ is elementary, we get $V_{<\omega}^M\models \forall \vec x (\varphi(\vec x,\vec a')\leftrightarrow \psi(\vec x,\vec b'))$.  Since $V_{<\omega}^M\preceq M'$, the same equivalence holds in $M'$.  Finally, if $\vec c \in V_\omega^M$, then $M\models \varphi_\omega(\vec c,\vec a')$ if and only if $M'\models \varphi_\omega(\vec c,\vec a')$ and likewise with $\psi_\omega$.
\end{proof}

We recall the notion of pseudofinite dimension, especially as considered in \cite{MR3091666,MR2833482}.  Since $M$ is an ultraproduct of finite sets, any definable set $D$ has a nonstandard cardinality $|D|$ in $\mathbb{R}^*$ (the ultrapower of the reals).  We let $\mathcal{C}$ be the convex hull of $\mathbb{Z}$ in $\mathbb{R}^*$.  Then for any definable set $X$, we can define
\[\delta_M(X)=\log|X|/\mathcal{C},\]
the image of $\log|X|$ in $\mathbb{R}^*/\mathcal{C}\cup\{-\infty\}$ (where $\log|X|=-\infty$ if $|X|=0$).  This is the \emph{fine pseudofinite dimension}.

The fine pseudofinite dimension satisfies the quasi-dimension axioms:
\begin{itemize}
\item $\delta_M(\emptyset)=-\infty$ and $\delta_M(X)>-\infty$ implies $\delta_M(X)\geq 0$,
\item $\delta_M(X\cup Y)=\max\{\delta_M(X),\delta_M(Y)\}$,
\item For any definable function $f:X\rightarrow Z$ and every $\alpha\in\mathbb{R}^*/\mathcal{C}\cup\{-\infty\}$, if $\delta_M(f^{-1}(z))\leq\alpha$ for all $z\in Z$ then $\delta_M(X)\leq\alpha+\delta_M(Z)$.
\end{itemize}

One of the features of fine pseudofinite dimension is that if we fix any definable set $X$, we may define a measure $\mu_X(Y)$ on definable $Y$ by $\mu_X(Y)=st(\frac{|Y|}{|X|})$ so that $\delta_M(Y)=\delta_M(X)$ if and only if $\mu_X(Y)\in(0,\infty)$.

 In light of the lemma above, the following definition makes sense. 
\begin{df}
Suppose $X\subseteq N^k$ is definable.  Without loss of generality, we may suppose that $X$ is defined by $\varphi(\vec x,\vec a)$, where $\varphi(\vec x,\vec y)$ is quantifier-free.  We then define $\delta_N(X)=\delta_M(\varphi_\omega(\vec x,\vec a'))$, where the latter dimension is computed in the pseudofinite structure $M$.
\end{df}

\begin{lemma}
  $\delta_N(X\times Y)=\delta_N(X)+\delta_N(Y)$
\end{lemma}
\begin{proof}
  Suppose $X$ and $Y$ are defined by $\varphi(\vec x,\vec a)$ and $\psi(\vec y,\vec b)$ respectively.  Then $X\times Y$ is defined by $\rho(\vec x,\vec y,\vec a,\vec b)=\varphi(\vec x,\vec a)\wedge\psi(\vec y,\vec b)$.  Then
\[\delta_N(X\times Y)=\delta_M(\rho_\omega(\vec x,\vec y,\vec a,\vec b))=\delta_M(\varphi_\omega(\vec x,\vec a))+\delta_M(\psi_\omega(\vec y,\vec b))=\delta_N(X)+\delta_N(Y).\]
using the pseudofinite axioms for $\delta_M$.
\end{proof}



$\delta_N$ need not satisfy the final quasi-dimension axiom, however---it is possible that there are many values $z\in Z_\omega$ so that $\delta(f^{-1}(z))$ is large and so $\delta_M(X_\omega)$ is large as well, but that none of these are in the image of $M$, so $\delta_N(X)$ is large even though $\delta_N(f^{-1}(z))$ is small for all $z\in Z$.




Nonetheless, there is a connection between $\delta_N$ and dividing, essentially the one shown by Garc\'ia in \cite{2014arXiv1402.5212G} for pseudofinite dimension.
\begin{prop}
Suppose that $\psi(x,a)$ and $\varphi(x,b)$ are existential $\la$-formulae with parameters from $V_\omega^M$ such that $\varphi(x,b)$ implies $\psi(x,a)$ and $\varphi(x,b)$ divides over $a$.  Then there is $b^\#\in V_\omega^M$ with $b^\#\equiv_{\la,a}b$ and $\delta_M(\varphi(x,b^\#))<\delta_M(\psi(x,a))$.
\end{prop}

\begin{proof}
Assume that no $b^\#$ exists as in the conclusion.  We then use that to get $K\in \mathbb N$ such that $K|\varphi_\omega(x,b^\#)|\geq |\psi_\omega(x,a)|$ for all $b^\#\in V_\omega^M$ with $b^\#\equiv_{\la,a}b$.  In fact, by saturation again, there is $\chi(x,a)\in \tp^M_{\la}(b/a)$ such that $K|\varphi_\omega(x,b^\#)|\geq |\psi_\omega(x,a)|$ for all $b^\#\models \chi_\omega(x,a)$.

Fix $L$ sufficiently large (depending only on $k$ and $K$) and take $(b_i)_{i<L}$ from $V_{<\omega}^M$ satisfying $\chi_\omega(x,a)$ 
and such that $\{\varphi(x,b_i) \ : \ i<L\}$ is $k$-inconsistent.  In particular, we have $K|\varphi_\omega(x,b_i)|\geq |\psi_\omega(x,a)|$ for all $i<L$.  As in \cite{2014arXiv1402.5212G}, if $L$ is sufficiently large, we get $i_1<\ldots<i_k<L$ such that $\bigcap_{j=1}^k \mu_{\psi_\omega}(\varphi_\omega(x,b_{i_j}))>0$.   In particular, there is $c\in V_\omega^M$ such that $M\models \varphi_\omega(c,b_{i_j})$ for all $j=1,\ldots,k$.  It follows that $M'\models \varphi(c,b_{i_j})$ for $j=1,\ldots,k$, a contradiction.
\end{proof}

In the previous result, if we have $\psi(x,a)$ and $\varphi(x,b)$ formulae with parameters from $N$ such that $\varphi(x,b)$ implies $\psi(x,a)$ and $\varphi(x,b)$ forks over $a$, then we can apply the previous result with $\psi(x,a')$ and $\varphi(x,b')$.  It should not be too surprising that, even in this situation, we need to look in $M'$ for the desired witness to dimension drop as $N$ is usually not saturated enough to see this dimension drop.

Combining the previous proposition with the remarks made in the previous paragraph yields the main result of this note:

\begin{thm}
  Suppose that $\psi(x,a)$ and $\phi(x,b)$ are existential $\la$-formulae with parameters from $N$ such that $\phi(x,b)$ implies $\psi(x,a)$ and $\phi(x,b)$ divides over $a$.  Then there is an elementary extension $M$ of $N$, an extension of $\delta_N$ to a quasidimension $\delta_M$ on $M$, and a $b^\#\in M$ with $b^\#\equiv_{\la,a}b$ and $\delta_M(\phi(x,b^\#))<\delta_M(\psi(x,a))=\delta_N(\psi(x,a))$.
\end{thm}

\begin{nrmk}
Note that a similar argument applies to an arbitrary relational language by taking $\la_0$ and an $\la_0$-structure $N$ and letting $T$ be the theory of the Morleyization of $N$.
\end{nrmk}

\begin{nrmk}
Note that the same construction applies, with only the obvious changes, to theories in continuous logic.
\end{nrmk}

\bibliographystyle{plain}
\bibliography{Pseudopseudofinite}
\end{document}